\documentclass{article}
\usepackage[T1]{fontenc}
\usepackage[english]{babel}
\usepackage{graphicx}
\usepackage{amssymb, amsmath, amsfonts}
\usepackage{dsfont}
\usepackage{vmargin}
\usepackage{color}
\usepackage{amsthm}       

\everymath{\displaystyle}

\usepackage{hyperref}
\hypersetup{
	colorlinks=true,
	breaklinks=true,
	urlcolor=blue,
	linkcolor=red,
	bookmarksopen=true,
	pdftitle = Asymptotic properties of stochastic Cahn-Hilliard equation with singular nonlinearity and degenerate noise,
	pdfauthor = Ludovic Gouden\`ege and Luigi Manca,
	pdfsubject = Asymptotic properties of stochastic Cahn-Hilliard equation with singular nonlinearity and degenerate noise,
	pdfkeywords = {asymptotic, Cahn-Hilliard, stochastic partial differential equations, singular nonlinearity, degenerate noise}
}

\numberwithin{equation}{section} 
\renewcommand{\thefootnote}{\*}

\newcommand{\R}{\mathbb{R}} 
\newcommand{\N}{\mathbb{N}} 
\newcommand{\E}{\mathbb{E}} 
\newcommand{\sL}{\mathrm{L}}
\newcommand{\sH}{\mathrm{H}}
\newcommand{\dd}{\mathrm{d}} 
\newcommand{\dds}{\mathrm{d}s}
\newcommand{\ddt}{\mathrm{d}t}
\newcommand{\ddtheta}{\mathrm{d\theta}}
\newcommand{\ddx}{\mathrm{d}x}

\def\P{\mathbb P}
\def\Tr{\mathrm{Tr}}

\def\ds{\displaystyle}

\newcommand{\BD}{\begin{displaymath}}
\newcommand{\ED}{\end{displaymath}}
\newcommand{\BEA}{\begin{eqnarray}}
\newcommand{\EEA}{\end{eqnarray}}
\newcommand{\BEAS}{\begin{eqnarray*}}
\newcommand{\EEAS}{\end{eqnarray*}}
\newcommand{\BE}{\begin{equation}}
\newcommand{\EE}{\end{equation}}
\newcommand{\BES}{\begin{equation*}}
\newcommand{\EES}{\end{equation*}}

\newcommand{\Proof}{\noindent {\bf{\underline{Proof} : }}}
\newcommand{\EProof}{\begin{flushright}$\Box$\end{flushright}}

\newtheorem{Def}{Definition}[section]

\newtheorem{Le}{Lemma}[section]
\newtheorem{Th}{Theorem}[section]
\newtheorem{Prop}{Proposition}[section]

\title{Asymptotic properties of stochastic Cahn-Hilliard equation with singular nonlinearity and degenerate noise}
\author{Ludovic Gouden\`ege\footnotemark[1], Luigi Manca\footnotemark[2]}
\date{}

\begin{document}

\maketitle
\renewcommand{\thefootnote}{\*}
\footnotetext[0]{\!\!\!\!\!\!\!\!\!\!\!AMS 2000 subject classifications. 60H15, 60H07, 37L40}
\renewcommand{\thefootnote}{\fnsymbol{footnote}}
\footnotetext[1]{CNRS, LAMA, 77454 Marne-la-Vall\'ee, France, goudenege@math.cnrs.fr}
\footnotetext[2]{LAMA, 77454 Marne-la-Vall\'ee, France, luigi.manca@u-pem.fr}

\begin{abstract}
We consider a stochastic partial differential equation with a logarithmic nonlinearity with singularities at $1$ and $-1$
and a constraint of conservation of the space average. 
The equation, driven by a trace-class space-time noise, contains a bi-Laplacian in the drift. 
We obtain existence of solution for equation with polynomial approximation of the nonlinearity.
Tightness of this approximated sequence of solutions is proved, leading to a limit transition semi-group.
We study the asymptotic properties of this semi-group, showing the existence and uniqueness of invariant measure,
asymptotic strong Feller property and topological irreducibility.

\end{abstract}

\section*{Introduction and main results}
The Cahn-Hilliard-Cook equation is a model to describe phase separation in a binary alloy (see \cite{C}, \cite{CH1} and \cite{CH2}) in the presence of thermal fluctuations (see \cite{COOK} and \cite{LANGER}). It takes the form
\BE
\label{Eq:0.1}
\left\{\begin{array}{ll}
\partial_{t} u = - \frac{1}{2}\Delta\left(\Delta u - \psi(u)\right)+\dot{\xi},&\text{ on } \Omega\subset \R^n,\\
\\
\nabla u \cdot \nu = 0 = \nabla (\Delta u-\psi(u)) \cdot \nu, &\text{ on } \partial\Omega,\\ 
\end{array}\right.
\EE
where $t$ denotes the time variable and $\Delta$ is the Laplace operator. Also $u\in [-1,1]$ represents the ratio between the two species and the noise term $\dot{\xi}$ accounts for the thermal fluctuations. The nonlinear term $\psi$ has the double-logarithmic form
\BE\label{Eq:0.2}
\psi : u \mapsto \frac{\theta}{2} \ln\left(\frac{1+u}{1-u}\right) - \theta_{c} u,
\EE
where $\theta$ and $\theta_{c}$ are temperatures with $\theta < \theta_{c}$.

The unknown $u$ represents the concentration of one specie with respect to the second one. In the deterministic case, the equation is obtained as a gradient in $\sH^{-1}(\Omega)$ of the free energy
$$
{\mathcal E}(u)=\int_{\Omega} \left(\frac12|\nabla u|^2 + \Psi(u)\right) \ddx 
$$
where $\Psi$ is an antiderivative of $\psi$. Since the gradient is taken in $\sH^{-1}(\Omega)$, a fourth order equation is obtained. A key point is that the average $\ds \int_\Omega u\  \ddx$ is conserved.

The deterministic equation where $\psi$ is replaced by a polynomial function has first been studied (see \cite{CH1}, \cite{LANGER} and \cite{MR763473}). The drawback of 	a polynomial nonlinearity is that the solution is not constrained to remain in the physically relevant interval $[-1,1]$. Singular nonlinear terms, such as the logarithmic nonlinearity considered here, remedy this problem. 
Such non smooth functions $\psi$ have also been considered (see \cite{MR1123143} and \cite{MR1327930}).

Phase separation has been analyzed thanks to this model: see for example the survey \cite{MR1657208}, and the references therein, or others recent results on spinodal decomposition and nucleation in \cite{MR1232163, MR2342011, MR1214868, MR1637817, MR1753703, MR1712442, MR1763320, MR2048517}. 

The noise term accounts for thermal fluctuations and is a commonly accepted model now. 
This stochastic equation has also been first studied in the case of a polynomial nonlinearity (see \cite{BLMAWA1, BLMAWA2, MR1867082, MR1897915, MR1359472, MR1111627}). Again, with a polynomial nonlinear term the solutions do not remain in $[-1,1]$ in general. It is even worse in the presence of noise since any solution leaves this interval immediately with positive probability.

Since the Cahn-Hilliard equation is a gradient flow in $\sH^{-1}(\Omega)$, it is natural to consider a noise which is a cylindrical Wiener process in this space ({\it i.e.} the spatial derivative of a space-time white noise). Another justification of this choice for the noise is that the stochastic equation has still a gradient structure and an explicit invariant measure is known.

Unfortunately, if the dimension $n$ is larger than $1$, it is not difficult to see that with such a noise, even for the linear equation, the solutions have negative spatial regularity and it does not seem possible to treat the nonlinearity. 
Finally, it is not expected that the singularity of logarithmic nonlinear term is strong enough to prevent 
the solution from exiting the interval $[-1,1]$. It has been rigorously proved in \cite{MR2568285} that a nonlinear term of the form $\ln u$ is not strong enough to ensure that solutions remain positive and a reflection measure has to be added. 


The stochastic heat equation with reflection, {\it i.e.} when the fourth order operator is replaced by the Laplace operator, is a model for the evolution of random interfaces near a hard wall. 
It has been extensively studied in the literature (see \cite{MR2257651}, \cite{MR1835843}, \cite{MR1872740},  \cite{MR1172940} \cite{MR1921014}, \cite{MR1959795} and \cite{MR2128236}). 
The two-walls case has been treated  in \cite{MR2277518}.
Essential tools in these articles are the comparison principle and the fact that the underlying Dirichlet form is symmetric so that the invariant measure is known explicitly. 

However, no comparison principle holds for fourth order equations and new techniques have to be developed. 
Equation \eqref{Eq:0.1} has been studied in \cite{MR2349572} with a single reflection and when no nonlinear term is taken into account. The reflection is introduced to enforce positivity of the solution. It has been shown in \cite{MR2421180} that a conservative random interface model close to one hard wall  gives rise to this fourth-order stochastic partial differential equation with one reflection. Various techniques have been introduced to  overcome the lack of comparison principle. 
Moreover, as in the second order case, an integration by part formula for the invariant measure has been 
derived. Then, in \cite{MR2568285}, a singular nonlinearity of the form $u^{-\alpha}$ or $\ln u$ has been considered. Existence and uniqueness of solutions have been obtained and using the integration by parts formula as in \cite{MR1959795}, it has been proved that the reflection measure vanishes if and only if $\alpha \ge 3$. In particular, as mentioned above, for a logarithmic nonlinearity, the reflection is active. For a double singular nonlinearity of the form \eqref{Eq:0.2}, thanks to a delicate a priori estimate on the $\sL^1$ norm of the nonlinearity, the authors have obtained uniqueness and existence.

In this article, we do not consider that the noise is a cylindrical Wiener process in $\sH^{-1}$.
In this case the equation has not the desired gradient structure and the technics developed in \cite{DEBUSSCHE-GOUDENEGE}, \cite{MR2349572} and \cite{MR2568285} are not any more valid.
However the choice of smooth-enough space-time noise permits to work with an It\^o's formula.
Following the approach  of \cite{DEBUSSCHE-GOUDENEGE}, we shall consider the problem with a polynomial approximation of the nonlinearity $\psi$.
We obtain {\em a priori} estimates on the approximated solutions which are sufficient to prove tightness.

Unfortunately, due to the degenerate noise, we are not able to show that the approximated solutions are tight in $\mathcal{C}([0,T]\times[0,1];[-1,1])$ and, as consequence,
we are not able to characterise the limit process $X(t,x)$ as a solution of some stochastic partial differential equation.

Despite of this, we discuss about the asymptotic properties of the transition semi-group $P_t$ associated to this process.
The main result of the paper consist in Theorems \eqref{prop.ex.inv.meas} and \eqref{thm.erg}, where we show that $P_t$ is ergodic and admits a unique invariant measure. 
The ergodicity result is  obtained by assuming that the noise  acts on the space spanned by the first eigenvectors of the laplacian  (see \eqref{assumptionB}, \eqref{eq.N}) ;
this property is sometimes called {\em essentially elliptic} (see \cite{MR2259251}, \cite{MR2786645}).

The paper is organised as follows : in the next section we introduce notations and the approximated problem ; in section \ref{S:2} we obtain some a priori estimate and the convergence of 
the approximated process to the limit process as well as the associated transition semigroup ; in section \ref{S:3} we discuss about the ergodicity properties of the transition semigroup.


\section{Preliminaries}\label{S:1}

We denote by $\langle\cdot,\cdot\rangle$ the scalar product in $\sL^2(0,1)$;  $A$ is the realization in $\sL^2(0,1)$ of the Laplace operator with Neumann boundary condition, i.e.
\BD
D(A) = \text{ Domain of }A  = \{ h \in W^{2,2}(0,1) : h'(0) = h'(1) = 0\}
\ED
where $W^{2,2}(0,1)$ is the classical Sobolev space. 
Remark that $A$ is self-adjoint on $\sL^2(0,1)$ and we have a complete orthonormal system of eigenvectors $(e_{i})_{i\in\N}$ in $\sL^2(0,1)$ for the eigenvalues $(\lambda_{i})_{i\in\N}$. We denote by $\bar{h}$ the mean of $h \in \sL^2(0,1)$
\BD
\bar{h} = \int_{0}^1h(\theta)\ddtheta.
\ED
We remark that $A$ is invertible on the space of functions with $0$ average. In general, we define
$(-A)^{-1}h=(-A)^{-1}(h-\bar h)+\bar h$.

For $\gamma \in\R$, we define $(-A)^{\gamma}$ by classical interpolation. We set 
$V_{\gamma}:=D((-A)^{\gamma/2})$. 
It is endowed with the classical seminorm and norm
\BD
|h|_{\gamma} =  \left(\sum_{i=1}^{+\infty} (-\lambda_{i})^{\gamma} h_{i}^2\right)^{1/2},\; 
\|h\|_{\gamma} =  \left(|h|_{\gamma}^2 + \bar{h}^2\right)^{1/2},\mbox{ for } h=\sum_{i\in\N} h_i e_i.
\ED
$|\cdot|_\gamma$ is associated to the scalar product $( \cdot,\cdot)_{\gamma}$.
To lighten notations, we set $(\cdot,\cdot) := (\cdot,\cdot)_{-1}$  and $\sH:=V_{-1}$. The average can be defined in any $V_\gamma$ by $\bar h= (h,e_0)$.
It plays an important role and we often work with functions with a fixed average $c\in \R$. We define $\sH_c=\{h\in \sH : \bar{h}=c\}$ for all $c\in \R$. 

We denote by $\mathcal{B}_b(\sH_{c})$ the space of all Borel bounded functions 

The covariance operator of the noise is an operator $B$ 
 such that 
\BES
\Tr_{\gamma} := \Tr\left[\sqrt{B}(-A)^{\gamma}\sqrt{B}^{*}\right]
\EES
is finite for some $\gamma \in\R$ which will be precise after. Moreover we assume $\sqrt{B} e_{0} = 0$ in order to ensure the conservation of average.

For $\lambda \in \R$, we define:
\BD                                                                                                                                                                                                                                                                                                      
f(u):=\left\{\begin{array}{lr}
+\infty,& \text{ for all } u \leq -1,\\
\\
\ln \left(\frac{1-u}{1+u}\right) + \lambda u,& \text{ for all } u \in (-1,1),\\
\\
-\infty,& \text{ for all } u \geq 1,
\end{array}\right.                                                                                                                                                                                                                                                                                                      
\ED
and the following antiderivative $F$ of $-f$:
\BD
F(u)= (1+u) \ln(1+u)+(1-u) \ln(1-u)-\frac{\lambda}2 u^2,\text{ for all }u \in (-1,1).
\ED
Let $X(\cdot,\cdot)$ stand for a function on $[0,T]\times[0,1]$. With these notations, we rewrite \eqref{Eq:0.1} in the abstract form:
\BE\label{Eq:1.1}
\left\{\begin{array}{l}
\dd X=-\frac{1}{2}A\left(AX+f(X)+\eta_{-}-\eta_{+}\right)\ddt + \sqrt{B} \ \dd W,\\
\\
\langle (1+X) ,\eta_{-}\rangle_{O_{T}} = \langle (1-X) ,\eta_{+}\rangle_{O_{T}} =0,\\
\\
X(0,\cdot)=x \text{ for } x \in \sH,
\end{array}\right.
\EE
where $W$ is a cylindrical Wiener process on $\sL^2(0,1)$ and
\BD
\big\langle v,\zeta\big\rangle_{O_T} = \int_{[0,T]\times [0,1]} v\ \dd\zeta.
\ED

The equation \ref{Eq:1.1} is characterized by the two reflection measures $\eta_{\pm}$. They act as a force to prevent the solution from leaving the physical domain $[-1,1]$.
They appears naturally because the logarithmic nonlinearity is not strong enough. It can also be seen as a Lagrange multiplier for the condition ``$X(t,x) \in [-1,1]$ for all $t>0$''.
This interpretation is strongly linked to the contact conditions $\langle (1+X) ,\eta_{-}\rangle_{O_{T}} = \langle (1-X) ,\eta_{+}\rangle_{O_{T}} =0$.
In the white noise case studied in \cite{DEBUSSCHE-GOUDENEGE}, the stationary solution has enough regularity to obtain a $\sL^{1}(O_{T})$ estimate on $f(X)$.
This estimates permits to give a sense to $\int_{0}^{T}\int_{0}^{1}f(X(t,\theta)) Ah(\theta)\ddt\dd\theta$ since $Ah \in \sL^{\infty}(O_{T})$.
Uniqueness can be proved under a nice definition of the weak solution as in \cite{DEBUSSCHE-GOUDENEGE}.

The solution of the linear equation with initial data $x \in \sH$ is given by
\BES
Z(t,\cdot,x) = e^{-tA^2/2}x + \int_0^te^{-(t-s)A^2/2}\sqrt{B}\ \dd W_s.
\EES
As easily seen this process is in $\mathcal{C}([0,+\infty[;\sH)$ (see \cite{MR1207136}). 
In particular, the mean of $Z$ is constant and the law of the process $Z$ is the Gaussian measure:
\BD
Z(t,\cdot,x)\sim\mathcal{N}\big(e^{-tA^2/2}x,Q_t\big),\;
Q_t = \int_0^t\sqrt{B}^{*}e^{-sA^2/2}e^{-sA^2/2}\sqrt{B}\ \dds = \sqrt{B}^{*}(-A)^{-1}(I - e^{-tA^2})\sqrt{B}.
\ED
If we let $t\rightarrow +\infty$, the law of $Z(t,\cdot,x)$ converges to the Gaussian measure on $\sL^2$:
\BD
\mu_c := \mathcal{N}(ce_0,\sqrt{B}(-A)^{-1}\sqrt{B}^{*}), \text{ where } c=\overline{x}.
\ED

In order to solve equation \eqref{Eq:1.1}, we use polynomial approximations of this equation. We denote by $\{f_{n}\}_{n \in \N}$ the sequence of polynomial functions which converges to the function $f$ on $(-1,1)$, defined for $n \in \N$ by: 
\BD
f_{n}(u)=-2\sum_{k=0}^{n}\frac{u^{2k+1}}{(2k+1)} + \lambda u, \text{ for all } u \in \R.
\ED
We also use in a crucial way 
that $u\mapsto f_{n}(u)-\lambda u$ is monotone non-increasing below.

\noindent Then for $n \in \N$, we study for the following polynomial approximation of \eqref{Eq:1.1} with an initial condition $x \in \sH$: 
\BE\label{Eq:1.2}
\left\{\begin{array}{l}
\dd X^n+\frac{1}{2}(A^2X^n+Af_{n}(X^n))\ddt = \sqrt{B}\ \dd W,\\
\\
X^n(0,\cdot)=x.
\end{array}\right.
\EE
This equation has been studied in \cite{MR1359472} in the case $B=I$.  The results generalize immediately and  it can be proved that for any $x\in \sH$, there exists a unique solution $X^n(\cdot,\cdot)$ $a.s.$ in 
$\mathcal{C}([0,T];\sH)\cap \sL^{2n+2}((0,T)\times (0,1))$. 
It is a solution in the mild or weak sense. Moreover the average of $X^n(t,\cdot)$ does not depend on $t$.

For each $c\in \R$, \eqref{Eq:1.2} defines a transition  semigroup $(P^{n,c}_t)_{t\ge 0}$:
\BES
P^{n,c}_t\phi(x) = \E[\phi(X^{n}(t,\cdot))], \; t \geq 0, x \in \sH_{c}, \; \phi \in \mathcal{B}_b(\sH_{c}),\; n \in \N.
\EES
Existence of an invariant measure can be proved as in \cite{MR1359472}.

%







In all the article, C denotes a constant which may depend on parameters and its value may change from one line to another.

\section{{\em A priori} estimates in Hilbert spaces and tightness}\label{S:2}

In this section we prove the tightness of the solutions of approximated equations, and obtain a limit transition semi-group.
Fix $-1<c<1$ and $x$ an initial data in $\sH_{c}$. Now consider the unique solution of \eqref{Eq:1.2} denoted $X^{n}$ in $\sH_{c}$ with initial data $x$.
We are going to prove that the laws of $(X^{n})_{n\in\N}$ are tight in some suitable space. First, we prove a result which only needs the assumption $\Tr_{-1} < +\infty$.
\begin{Prop}\label{Prop:3.1}
Suppose $\Tr_{-1}<+\infty$. The laws of $(X^{n})_{n\in\N}$ are tight in $\sL^{\infty}([0,T];V_{-1}) \cap \sL^{2}([0,T];V_{1})$, and we have
\BE\label{Eq}
\E\left[\int_{0}^{T}|X^{n}(t)|_{1}^2\ddt\right]\leq |x|_{-1}^2+T\mathcal{Q}_{c}(\lambda),
\EE
where $\mathcal{Q}_{c}(\lambda)$ is a polynomial.
\end{Prop}
\Proof  
Applying It\^o formula to $|X^{n}(t)|_{-1}^2$, we obtain
\BEAS
|X^{n}(T)|_{-1}^2-|X^{n}(0)|_{-1}^2+\int_{0}^{T}|X^{n}(t)|_{1}^2\ddt 
+2\int_{O_T} \sum_{k=0}^{n}\frac{(X^{n})^{2k+2}}{2k+1}  \dds\ddtheta
+c\int_{O_T} \lambda X^{n} \ \dds\ddtheta
\\
= 2 \int_{0}^{T}( X^{n}(t),\sqrt{B}\ \dd W_{t}) +   T\, \Tr_{-1}
+\int_{O_T} \lambda (X^{n})^2\dds\ddtheta
+2c\int_{O_T} \sum_{k=0}^{n}\frac{(X^{n})^{2k+1}}{(2k+1)}  \dds\ddtheta
\EEAS

Using H\"older's inequality, we have   
\BEAS
\sum_{k=0}^n c \int_{0}^1 \frac{(X^n(t))^{2k+1}}{2k+1} \ddtheta&\leq&
\sum_{k=0}^n \frac{c}{2k+1}\left(\int_{0}^1 (X^n(t))^{2k+2} \ddtheta\right)^{2k+1/2k+2}\\
&\leq&
\sum_{k=0}^n \left(\frac{1}{2k+2}\int_{0}^1 (X^n(t))^{2k+2} \ddtheta+\frac{c^{2k+2}}{(2k+1)(2k+2)}\right).
\EEAS
And since 
\BES
\sum_{k=0}^n \int_{0}^1 \frac{(X^n(t))^{2k+2}}{2k+1} \ddtheta-\sum_{k=0}^n \int_{0}^1 \frac{(X^n(t))^{2k+2}}{2k+2} \ddtheta=\sum_{k=0}^n \int_{0}^1 \frac{(X^n(t))^{2k+2}}{(2k+1)(2k+2)} \ddtheta,
\EES
we obtain
\BEAS
|X^n(T)|_{-1}^2-|X^n(0)|_{-1}^2+\int_{0}^{T}|X^n(t)|_{1}^2\ddt 
+2\int_{O_T} \sum_{k=0}^{n}\frac{(X^n(t))^{2k+2}}{(2k+1)(2k+2)}  \ddt\ddtheta
+c\int_{O_T} \lambda X^n(t) \ \ddt\ddtheta
\\
\leq 2 \int_{0}^{T}(X^n(t),\sqrt{B}\ \dd W_{t}) + T\, \Tr_{-1}
+\int_{O_T} \lambda (X^n(t))^2\ddt\ddtheta
+2T\sum_{k=0}^n \frac{c^{2k+2}}{(2k+1)(2k+2)}.
\EEAS
Using $\overline{X^n(t)} = c$ for all $t\geq0$, we obtain
\BEA\label{Eq:2.1}
|X_c^n(T)|_{-1}^2-|X^n(0)|_{-1}^2+\int_{0}^{T}|X^n(t)|_{1}^2\ddt+ c^2 \lambda T\nonumber\\
+\int_{O_T} \frac{(X^n(t))^4}{6} \ddt\ddtheta+(1-\lambda)\int_{0}^T|X^n(t)|_{0}^2 \ddt
+\int_{O_{T}} 2\sum_{k=2}^n \frac{(X^n(t))^{2k+2}}{(2k+1)(2k+2)} \ddt\ddtheta
\\
\leq 2 \int_{0}^{T}(X^n(t),\sqrt{B}\ \dd W_{t}) + T\, \Tr_{-1} +2T\sum_{k=0}^n \frac{c^{2k+2}}{(2k+1)(2k+2)}.\nonumber
\EEA
Remark that
\BES
\frac{y^4}{6}+(1-\lambda) y^2 + \frac{3}{2}(1-\lambda)^2 \geq0,
\EES
for any real number $y$. Let $y=X^n(t)$ and integrate on $O_{T}$, it follows that
\BE\label{Eq:2.2}
\int_{O_{T}} \frac{(X^n(t))^4}{6} \ddt\ddtheta +(1- \lambda)\int_{0}^T |X^n(t)|_{0}^2 \ddt+ \frac{3}{2}T(1-\lambda)^2 \geq0.
\EE
Combining \eqref{Eq:2.1} and \eqref{Eq:2.2}, we obtain
\BEAS
|Xc^n(T)|_{-1}^2-|X^n(0)|_{-1}^2+\int_{0}^{T}|X^n(t)|_{1}^2\ddt
+\int_{O_{T}} 2\sum_{k=2}^n \frac{(X^n(t))^{2k+2}}{(2k+1)(2k+2)} \ddt\ddtheta
\\
\leq 2 \int_{0}^{T}(X^n(t),\sqrt{B}\ \dd W_{t}) +   T\, \Tr_{-1} +\frac{3}{2}T(1-\lambda)^2- c^2 \lambda T+2T\sum_{k=0}^n \frac{c^{2k+2}}{(2k+1)(2k+2)}.
\EEAS
Or simply
\BEAS
|X^n(T)|_{-1}^2-|Xn(0)|_{-1}^2+\int_{0}^{T}|X^n(t)|_{1}^2\ddt&\leq& 2 \int_{0}^{T}(X^n(t),\sqrt{B}\ \dd W_{t})\\
&+& T\left( \Tr_{-1} +\frac{3}{2}(1-\lambda)^2- c^2 \lambda +F(c)\right).
\EEAS 
Remark that for all $c\in(-1,1)$ the polynomial function $\mathcal{P}_{c} : \lambda \mapsto\frac{3}{2}(1-\lambda)^2- c^2 \lambda +F(c)$ is always nonnegative. Indeed, if we compute its classical discriminant, we find
\BEAS
\Delta_{\mathcal{P}_{c}} &=& (c^2+3)^2-6\left(\frac{3}{2}+F(c)\right)\\
&=&c^4+6c^2+9-9-12\sum_{k=0}^{+\infty}\frac{c^{2k+2}}{(2k+1)(2k+2)}\\
&=&-12\sum_{k=2}^{+\infty}\frac{c^{2k+2}}{(2k+1)(2k+2)}\leq0.
\EEAS
So the minimum of $\mathcal{P}_{c}$ is attained in $\lambda^* = \frac{c^{2}}{3}+1$ such that
\BD
\mathcal{P}_{c}(\lambda^*) = 2\sum_{k=2}^{+\infty}\frac{c^{2k+2}}{(2k+1)(2k+2)} \geq 0.
\ED
Denote $\mathcal{Q}_{c}(\lambda) = \Tr_{-1}+\mathcal{P}_{c}(\lambda)$. Using Poincar\'e's inequality and taking the expectation, we obtain
\BEAS
\E\left[|X^n(T)|_{-1}^2-|Xc^n(0)|_{-1}^2+\pi^4\int_{0}^{T}|X^n(t)|_{-1}^2\ddt\right]&\leq&T\mathcal{Q}_{c}(\lambda).
\EEAS
And the Gronwall lemma implies for all $t\in [0,T]$
\BEAS
\E\left[|X^n(t)|_{-1}^2\right] \leq \left(\E\left[|X^n(0)|_{-1}^2\right]-\frac{\mathcal{Q}_{c}(\lambda)}{\pi^4}\right)\exp(-\pi^4t) + \frac{\mathcal{Q}_{c}(\lambda)}{\pi^4}.
\EEAS
Under our assumption on the trace class property of the operator $\sqrt{B}$, there there exists a constant $C$ such that
\BEAS
\E\left[ \left(\int_{0}^{T} (X^n(t), \sqrt{B}\ \dd W_{t})\right)^2 \right]&=& \E\left[\int_{0}^{T} | \sqrt{QB}^{*}\  X^n(t)|^2_{-1}\ddt\right]\\
&\leq&C\  \E\left[\int_{0}^{T} | X^n(t)|^2_{-1}\ddt\right].
\EEAS
We set 
\BEAS
\varphi_{n} &=& |X^n(T)|_{-1}^2-| X^n(0)|_{-1}^2+\int_{0}^{T}| X^n(t)|_{1}^2\ddt
-T\mathcal{Q}_{c}(\lambda)\\
&&-\int_{O_T} f_{n_{k}}\left(X^n(t)\right)\left(X^n(t) -c\right) \ddt\ddtheta
\EEAS
such that
\BEAS
M^{2}\P(|\varphi_n|\geq M)&\leq&\E\left[\phi_{n}^{2}\right] \leq \E\left[\left(2 \int_{0}^{T}(X^n(t),\sqrt{B}\ \dd W_{t})\right)^{2}\right]\\
&\leq & C\  \E\left[\int_{0}^{T} | X_c^n(t)|^2_{-1}\ddt\right]\\
&\leq& C \left(\left(\E\left[| X_c^n(0)|_{-1}^2\right]-\frac{\mathcal{Q}_{c}(\lambda)}{\pi^4}\right)\left(1-\exp(-\pi^4T)\right) + T\mathcal{Q}_{c}(\lambda)\right).
\EEAS

we obtain
\BES
|X^{n}(T)|_{-1}^2-|X^{n}(0)|_{-1}^2+\int_{0}^{T}|X^{n}(t)|_{1}^2\ddt\leq 2 \int_{0}^{T}(X^{n}(t),\sqrt{B}\ \dd W_{t})+T\mathcal{Q}_{c}(\lambda),
\EES
giving the boundedness of expectation in $\sL^{2}([0,T];V_{1})$
\BES
\E\left[\int_{0}^{T}|X^{n}(t)|_{1}^2\ddt\right]\leq \E\left[|X^{n}(0)|_{-1}^2\right]+T\mathcal{Q}_{c}(\lambda).
\EES
Using Poincar\'e's inequality, and taking the expectation, we obtain
\BEAS
\E\left[|X^{n}(T)|_{-1}^2\right]+\pi^{4}\E\left[\int_{0}^{T}|X^{n}(t)|_{-1}^2\ddt\right]&\leq& \E\left[|X^{n}(0)|_{-1}^2\right] +T\mathcal{Q}_{c}(\lambda).
\EEAS
Using Gronwall lemma implies for all $t\in[0,T]$
\BES
\E\left[|X^{n}(t)|_{-1}^2\right] \leq \left(\E\left[|X^{n}(0)|_{-1}^2\right]-\frac{\mathcal{Q}_{c}(\lambda)}{\pi^4}\right)\exp(-\pi^4t) + \frac{\mathcal{Q}_{c}(\lambda)}{\pi^4},
\EES
giving the desired boundedness of expectation in $\sL^{\infty}([0,T];V_{-1}) \cap \sL^{2}([0,T];V_{1})$.\EProof

\begin{Prop}\label{Prop:3.2}
Suppose $x \in V_{0}$ and $\Tr_{0}<+\infty$. Then the laws of $(X^{n})_{n\in\N}$ are tight in $\sL^{2}([0,T];V_{2}) \cap \sL^{\infty}([0,T];V_{0})$, and we have
\BES
2\int_{O_T} (\nabla X^{n}(t,\theta))^{2} \sum_{k=0}^{n}(X^{n}(t,\theta))^{2k}\ddt \ddtheta \leq 2\int_{0}^{T}\langle\sqrt{B}^{*}X^{n}(t),\dd W_{t}\rangle + |x|_{0}^{2} + T\,\Tr_{0}.
\EES
\end{Prop}
\Proof
Applying Ito formula to $|X^{n}(t)|_{0}^2$, we obtain
\BEAS
|X^{n}(T)|_{0}^2-|X^{n}(0)|_{0}^2+\int_{0}^{T}|X^{n}(t)|_{2}^2\ddt 
+2\int_{O_T} \nabla X^{n}(t) \sum_{k=0}^{n}\frac{\nabla ((X^{n}(t))^{2k+1})}{2k+1}  \ddt\ddtheta
\\
= - \lambda\int_{O_T} (\nabla X^{n})^2\dds\ddtheta +
2 \int_{0}^{T}\langle X^{n}(t),\sqrt{B}\ \dd W_{t}\rangle +   T\, \Tr_{0}\\
 = |X^{n}(T)|_{0}^2-|X^{n}(0)|_{0}^2+\int_{0}^{T}|X^{n}(t)|_{2}^2\ddt 
+4\int_{O_T} (\nabla X^{n}(t))^{2} \left(\sum_{k=0}^{n}(X^{n}(t))^{2k}\right)\ddt\ddtheta.
\EEAS
Eliminating the positive terms, we obtain the desired estimation. Moreover taking the expectation, we obtain
\BEAS
\E\left[\int_{0}^{T}|X^{n}(t)|_{2}^2\ddt\right]&\leq& \E\left[|X^{n}(0)|_{0}^2\right] 
+ T\, \Tr_{0}.
\EEAS
Using Poincar\'e's inequality, taking the expectation, and using Gronwall lemma implies for all $t\in [0,T]$
\BEAS
\E\left[|X^{n}(t)|_{0}^2\right] \leq \left(\E\left[|X^{n}(0)|_{0}^2\right]-\frac{\Tr_{0}}{\pi^4}\right)\exp(-\pi^4t) + \frac{\Tr_{0}}{\pi^4}.
\EEAS\EProof

By Proposition \ref{Prop:3.1} and classical diagonalization procedure, we obtain
\begin{Prop}\label{Prop:3.3}
For all $c\in (-1,1)$, there exists a transition semi-group $P^{c}_t$, $t\geq0$ on $\mathcal{B}_c(\sH)$  and a subsequence $\{n_k\} $ such  that
$P^{n_k,c}_t\varphi(x) \to P^{c}_t\varphi(x)$ as $n_k\to \infty$, for all $\varphi \in \mathcal{B}_c(\sH)$, $x\in \sH$, $t\geq0$.
\end{Prop}

\section{Ergodicity properties} \label{S:3}
Existence of an invariant measure follows by a classical argument of compactness and by the estimate
\begin{equation}  \label{eq.lip}
   |X^n(t,x)-X^n(t,y)|_{-1}\leq e^{\lambda t}\|x-y\|_{-1}.
\end{equation}
which holds any $x,y\in \sH_c$, $t>0$.
\begin{Th} \label{prop.ex.inv.meas}
 For any  $c\in (-1,1)$ there exists an invariant measure for the semigroup $P^c_t$, $t\geq0$.
\end{Th}
\begin{proof}
 Fix $c\in (-1,1)$. 
 By \eqref{Eq} and by the Krylov-Bogoliubov criterion, for any $n\in\N$   there exists an invariant measure for the transition semigroup $P^{c,n}$,
 $t\geq0$. 
 Let  $\nu$ be as a weak${}^*$ limit of the sequence $\{\nu^n\}$.
 By using \eqref{eq.lip} and the thightness of $\nu^n$ 
 it is strightforward  to verify that $\nu$ is invariant for $P_t$.
\end{proof}


In order to show uniqueness of such invariant measure, 
we shall show (under some suitable assumption on the noise, see \eqref{assumptionB}) that its transition semigroup $P_t^c$ enjoys the asymptotic strong Feller property and it is topologically irreducible.
We recall that since the noise in degenerate, we are not able to use the same techniques of \cite{MR2568285}; 
in particular, the Bismut-Elworthy formula does not apply, and we are not able to show that the semigroup $P_t^c$ is strong Feller.
 
Asymptotic strong Feller property was introduced in \cite{MR2259251} in order to study ergodicity properties of a $2D$ Navier-Stokes equation perturbed by a very degenerate noise.
The exact definition can be found in \cite[Definition 3.8]{MR2259251}. 
Here we consider the property in the following form, which shall be sufficient for our pourpose:
\begin{Prop}[Proposition 3.12 in \cite{MR2259251}]
Let $t_n$ and $\delta_n$ be two positive sequences with $\{t_n\}$ nondecreasing and $\{\delta_n\}$ converging to zero.  
A semigroup $P_t$ on a Hilbert space $H$ is asymptotically strong Feller if, for all $\varphi:H\to\R$ with $|\varphi|_\infty$ and $|\nabla \varphi|_\infty$ finite,
\[
   |\nabla P_t\varphi(x)|\leq C(\|x\|_H)(|\varphi|_\infty+\delta_n|\nabla \varphi|_\infty)
\]
for all $n$, where $C:\R^+\to\R^+$ is a fixed nondecrasing  function.
\end{Prop}
Unfortunately, we are not able to show that the semigroup $P^c_t$ is differentiable (in the case of a white noise perturbation the semigroup is differentiable, and the strong Feller property holds, see \cite{MR2568285}).
Also for the semigroup $P_t^{c,n}$, associated to the approximation \eqref{Eq:1.2} we are not able to show its differentiability.
However, we shall use the following variation of the property
\begin{Prop} \label{prop.def.ASF}
 A semigroup $P_t$ on a Hilbert space $H$ is asymptotically strong Feller if there exist two constant $C, \delta>0$ such that for all $\varphi:H\to\R$ with $|\varphi|_\infty$ and $|\nabla \varphi|_\infty$ finite,
\[
   |P_t\varphi(x)-P_t\varphi(y)|\leq C(|x|_{-1}\vee|y|_{-1})(|\varphi|_\infty+e^{-\delta t}|\nabla \varphi|_\infty)|x-y|_H.
\]
\end{Prop}
This proposition can be  proved with the same argument used to prove Proposition 3.12 in \cite{MR2259251}.
The right-hand side is clearly derived by the previous proposition, setting $\delta_n=e^{-\delta t}$. 
The only modification consists in the left-side part, where no assumptions on differentiability of $P_t\varphi$ are needed.

The main assumption of this section is
\BE \label{assumptionB}
   B=\sum_{k=1}^\infty b_k \langle \cdot , e_k\rangle e_k
\EE
where $b_k>0$ for $k\in\{1,\ldots,N\}$ and
\BE \label{eq.N}
  \frac12(N+1)^2-\lambda>0. 
\EE
In other words, we assume that for a sufficiently large $N$, $\text{span}\{e_0,\ldots,e_N\}\subset \text{range}(B)$.
This kind of setting is known as \emph{essentially elliptic} (see \cite{MR2568285}).
We think that more degenerate noises can be considered in order to have asymptotic strong Feller property, and this shall be the objet of a forthcoming paper. 

The importance of the asymptotic strong Feller property is that in this case any two distinct ergodic invariant measures have disjoint topological support (see Theorem 3.16 in \cite{MR2259251}).
According to \cite{MR2786645}, we introduce the following 
\begin{Def}
 We say that a semigroup $P_t$ on a Hilbert space $H$ is \emph{weakly topologically irreducible} if for all $x_1, x_2\in H$ there exists $y\in H$ so that for any open set $A$ containing $y$ there exist $t_1,t_2>0$ with
 $P_{t_i}(x_i,A)>0$.  
\end{Def}
The main result of this section is the following
\begin{Th} \label{thm.erg}
 Under assumptions \eqref{assumptionB} and \eqref{eq.N}, for any $c\in (-1,1)$ the semigroup $P_t^c,t\geq0$ has an unique invariant measure.
\end{Th}
\begin{proof}
By Corollary 1.4 from \cite{MR2786645}, a Markov semigroup which is Feller,  weakly topologically irreducible and asymptotically strong Feller admits at most 
one invariant probability measure. The proof then follows by Proposition \ref{prop.ASF} and Proposition \ref{prop.weak1}.
\end{proof}

%
%
In order to proceed, we need some apriori estimates on the solution of the approximated problem.
Le us denote by $X^n(t,x,W)$ the solution of \eqref{Eq:1.2} perturbed by the Wiener process $W$.

\begin{Le} \label{lemma.ASF}
There exist $\delta>0$ and a continuous non decreasing function $C:\R^+\to \R^+$ such that any $x,y\in \sH_c$, $n\in\N$ there exists a process $w_n\in L^2(0,\infty;H)$, adapted to $W$, satifying
\BE  \label{eq.W+w}
   \Big|X^n\Big(t,x,W+\int_0^\cdot w_n(s)ds\Big)-X^n(t,y,W)\Big|_{-1}\leq e^{-\delta t}|x-y|_{-1} 
\EE
and
\BE \label{eq.Ew_n}
  \E\big|1-e^{-\frac12\int_0^tw_n^2(s)ds+\int_0^tw_n(s)dW(s)}\big| \leq C(|x|_{-1}\vee |y|_{-1})|x-y|_{-1}.
\EE
\end{Le}
\begin{proof}
Fix $n\in\N$. 
Let us denote by $\pi_l$ the projection of $\sH_c$ into $\text{span}\{e_0,\ldots,e_N\}$ (the \emph{low} frequences)
 and by $\pi_h=I-\pi_l$ the projection into its orthogonal complement (the \emph{high} frequences). 
Let $\tilde X$ be the solution of 
\[
  \begin{cases} 
   d \tilde X= \big(-A^2\tilde X + \lambda A \pi_h \tilde X-A p_n(\tilde X) + \lambda A \pi_l X^n(t,y,W)\big)dt + BdW(t)\\
   \tilde X(0) = x.
  \end{cases}
\]
Here $p_n$ is the increasing polynomial $f_n-\lambda$.
Arguing as in the previous section, it is strightforward to show that this equation admits a unique solution in $C([0,T];\sH_c)\cap L^2([0,T];V_1)$, $T>0$, adapted to $W(t), t\geq0$. 
Now  set $w_n(t)=B^{-1}A\pi_l(\tilde X - X^n(t,y,W ))$. 
It is clear that the operator $B^{-1}A\pi_l:\R^N\to\R^N$ is well defined (thanks to \eqref{assumptionB}, \eqref{eq.N}) and that $w_n(t)$ is square integrable and adapted to $W(t),t\geq0$.
Moreover, we find that $X^n\big(t,x,W+\int_0^\cdot w_n(s)ds\big)$  satisfies the same equation of $\tilde X$, so that it has to coincide with $\tilde X$.

Consequently, $Y(t):=X\big(t,x,W+\int_0^\cdot w_n(s)ds\big)-X(t,y,W)$ solves
\[
 \begin{cases}
 dY(t)=\big(-A^2Y(t)+\lambda A\pi_h Y(t)-Ap_n\big(X(t,x,W+\int_0^\cdot w_n(s)ds)\big)+Ap_n(X(t,y,W)\big)dt\\
 Y(0)=x-y. 
 \end{cases} 
\]
By taking the scalar product of both members with $A^{-1}Y(t)$  and using the dissipativity of $p_n$ we find
\begin{multline*}
 \frac12 d|Y(t)|_{-1}^2=-|Y(t)|_1^2+\lambda \|\pi_h Y(t)\|^2-\big(p_n\big(X^n(t,x,W+\int_0^\cdot w_n(s)ds)\big)-p_n(X^n(t,y,W),Y(t)\big)\\
  \leq -\big(\frac12(N+1)^2-\lambda\big)|\pi_hY(t)|_0^2-|\pi_lY(t)|_0^2.
\end{multline*}
Here we have used the estimate
\BES
  | \xi |_1^2=|\pi_l\xi|_1^2+|\pi_h\xi|_1^2=|\pi_l\xi|_1^2+\sum_{k=N+1}^\infty k^2(\xi_k)^2\geq |\pi_l\xi|_0^2+(N+1)^2|\pi_h\xi|_0^2
\EES
where $\xi\in \sH_0$ and $\xi_k$ is the $k-$th Fourier mode.
Setting $\delta=2\pi\min\big\{ \frac12(N+1)^2-\lambda,1\big\}$ (which is stricly positive by \eqref{eq.N}) we deduce, by Gronwall lemma,
\[
 |Y(t)|_{-1}^2+\pi\delta\int_0^t\|Y(s)\|^2ds \leq |x-y|_{-1}^2.
\]
Moreover, since $ V_0\subset V_{-1}$, with $|x|_{-1}\leq \pi^{-1}\|x\|$ we deduce
\[
  |Y(t)|_{-1}\leq e^{-\delta t}|x-y|_{-1}.
\]
This proves \eqref{eq.W+w}. 
Let us show \eqref{eq.Ew_n}.
By the definition of $w_n$ and by \eqref{eq.W+w} we have
\BEA  \label{eq.w_n.0}
  |w_n(s)|\leq \|B^{-1}A\pi_l\|_{\mathcal L(\R^N;\R^N)}|\tilde X - X^n(s,y,W )|_{-1}\leq \|B^{-1}A\pi_l\|_{\mathcal L(\R^N;\R^N)}e^{-\delta s}|x-y|_{-1}.
\EEA
Notice that this estimate holds in $L^\infty(\Omega)$.
Then the process
\[
  e^{-\frac12\int_0^tw_n^2(s)ds+\int_0^tw_n(s)dW(s)}
\]
is a square integrable martingale and by the well known properties of the exponential martingales it holds
\begin{eqnarray} 
 && E\left|1-e^{-\frac12\int_0^tw_n^2(s)ds+\int_0^tw_n(s)dW(s)}\right|
 =  \E\left|\int_0^t e^{-\frac12\int_0^sw_n^2(\tau)d\tau+\int_0^sw_n(\tau)dW(\tau)}w_n(s)dW(s)\right|  \notag
\\
&&\qquad  \leq  \left(\E \int_0^t e^{- \int_0^sw_n^2(\tau)d\tau+2\int_0^sw_n(\tau)dW(\tau)}w_n^2(s)ds  \right)^{\frac12}   \notag
\\
 &&\qquad  \leq  \left( \int_0^t\E\left[e^{- \int_0^sw_n^2(\tau)d\tau+2\int_0^sw_n(\tau)dW(\tau)}\right]|w_n(s)|_{L^\infty(\Omega)}^2ds   \right)^{\frac12}   \notag
\\
&&\qquad  \leq  \left(\int_0^t\E\left[e^{-2\int_0^sw_n^2(\tau)d\tau
     +2\int_0^sw_n(\tau)dW(\tau)}\right]e^{\int_0^s|w_n(\tau)|_{L^\infty(\Omega)}^2d\tau}|w_n(s)|_{L^\infty(\Omega)}^2ds  \right)^{\frac12} \notag
\\
&&\qquad    =\left(\int_0^t e^{\int_0^s|w_n(\tau)|_{L^\infty(\Omega)}^2d\tau}|w_n(s)|_{L^\infty(\Omega)}^2ds  \right)^{\frac12}  \notag
\\ 
&&\qquad    =\left(e^{\int_0^t|w_n(\tau)|_{L^\infty(\Omega)}^2d\tau}-1 \right)^{\frac12}  \notag
\\
&&\qquad    \leq e^{\frac12\int_0^t|w_n(\tau)|_{L^\infty(\Omega)}^2d\tau}   \left(   \int_0^t|w_n(\tau)|_{L^\infty(\Omega)}^2d\tau  \right)^{\frac12}.    \notag  
\end{eqnarray}
By \eqref{eq.w_n.0} we have  
\begin{eqnarray*}
  \int_0^t|w_n(\tau)|_{L^\infty(\Omega)}^2d\tau  &\leq&   \|B^{-1}A\pi_l\|_{\mathcal L(\R^N;\R^N)}^2|x-y|_{-1}^ 2\int_0^t    e^{-2\delta \tau}d\tau
\\
  &\leq&  \frac{\|B^{-1}A\pi_l\|_{\mathcal L(\R^N;\R^N)}^2}{2\delta}|x-y|_{-1}^2.
\end{eqnarray*}
Then \eqref{eq.Ew_n} is satisfied with 
 \[
    C(t)=\frac{\|B^{-1}A\pi_l\|_{\mathcal L(\R^N;\R^N)}}{\sqrt{2\delta}}\exp\left(\frac{\|B^{-1}A\pi_l\|_{\mathcal L(\R^N;\R^N)}^2}{2\delta}t^2\right),\qquad t\geq0. \qedhere
 \]
\end{proof}
The previous lemma means that  we can find a Wiener process $W' $ such that
the solution of equation \eqref{Eq:1.2} starting by $x$ driven by the Wiener process $W'$ approches the solution of \eqref{Eq:1.2} starting by $y$ as $t\to \infty$. 
Moreover, this perturbation is uniformly controlled, in the sense of \eqref{eq.Ew_n}, by a constant independent by $n$.


\begin{Prop} \label{prop.ASF}
Assume that \eqref{assumptionB}, \eqref{eq.N} hold. Then for any $c\in (-1,1)$ the semigroup $P_t^c$ on $\sH_c$  has the asymptotic strong Feller property.
\end{Prop}
\Proof
Fix $c\in (-1,1)$. 
Taking into account Proposition \ref{prop.def.ASF}, it is sufficient to show that 
there exist a nondecreasing function $C:\R^+\to\R^+$ and a
a constant $\delta>0$ such that for any $\psi:\sH_c\to\R$ continuous and bounded with bounded Fr\'echet derivative it holds 
\BE \label{eq.ASF}
    |P_t^{c}\psi(x)-P_t^{c}\psi(y)|\leq C(|x|_{-1}\vee |y|_{-1})(|\psi|_\infty +e^{-\delta t}|\nabla \psi|_\infty)|x-y|_{-1}.
\EE
Taking into account that $P_t^{c,n}\stackrel{n\to\infty}{\longrightarrow} P_t^c$ (see Proposition \ref{Prop:3.3}), for any $\varepsilon>0$ there exists $n\in \N$ (depending on $\varepsilon$, $\psi$, $x$ and $y$) such that
\[
   |P_t^{c}\psi(x)-P_t^{c}\psi(y)|\leq \varepsilon + |P_t^{c,n}\psi(x)-P_t^{c,n}\psi(y)|.
\]
By the previous Lemma, we take $w_n(\cdot)$ such that \eqref{eq.W+w}, \eqref{eq.Ew_n} hold.
Then,
\begin{multline*}
  P_t^{c,n}\psi(x)- P_t^{c,n}\psi(y)=\E\psi(X^n(t,x,W))-\E\psi(X^n(t,y,W))=\E\psi(X^n(t,x,W))\\
      -\E\psi(X^n(t,x,W+\int_0^\cdot w_n(s)ds))+\E\psi(X^n(t,y,W+\int_0^\cdot w_n(s)ds))-\E\psi(X^n(t,y,W)).
\end{multline*}
We stress that the expectation is taken with respect to the Wiener process $W(t),t\geq0$. 
Then, by Girsanov theorem, we obtain
\begin{multline*}
  \E\psi(X^n(t,x,W))-\E\psi(X^n(t,x,W+\int_0^\cdot w_n(s)ds))\\
    =\E[\psi(X^n(t,x,W))(1-e^{-\frac12\int_0^tw_n^2(s)ds+\int_0^tw_n(s)dW(s)})].
\end{multline*}
Then by \eqref{eq.Ew_n} we deduce
\begin{multline*}
 |\E\psi(X^n(t,x,W))-\E\psi(X^n(t,x,W+\int_0^\cdot w_n(s)ds))|\\
\leq |\psi|_{L^\infty}\E|1-e^{-\frac12\int_0^tw_n^2(s)ds+\int_0^tw_n(s)dW(s)}|\leq  C(|x|_{-1}\vee |y|_{-1})|\psi|_{L^\infty}|x-y|_{-1}.
\end{multline*}
On the other side, by \eqref{eq.W+w} we obtain
\begin{multline*}
   |\E\psi(X^n(t,y,W+\int_0^\cdot w_n(s)ds))-\E\psi(X^n(t,y,W))|
\\
  \leq |\nabla\psi|\E|X^n(t,y,W+\int_0^\cdot w_n(s)ds)-X^n(t,y,W)|\leq |\nabla\psi|e^{-\delta t}|x-y|_{-1}.
\end{multline*}
Therefore it holds
\[
    |P_t^{c,n}\psi(x)- P_t^{c,n}\psi(y)|\leq C(|x|_{-1}\vee |y|_{-1})|\psi|_{L^\infty}|x-y|_{-1}+|\nabla\psi|e^{-\delta t}|x-y|_{-1}
\]
where the continuous  nondecreasing function $C:\R^+\to\R^+$ and the positive constant   $\delta$ are independent by $n$ and by $\varepsilon$. 
Consequently, \eqref{eq.ASF} follows.
\EProof
\begin{Prop} \label{prop.weak1}
 For any $c\in (-1,1)$, the semigroup $P_t^c,t\geq0$ on $\sH_c$ is weakly topologically irreducible.
\end{Prop}
\begin{proof}
We shall show that that for any $x\in \sH_c$, $\delta>0$ there exists $\varepsilon>0$ and $t>0$ such that $P^c_t(x,\overline{B(0,\delta)})\leq 1-\varepsilon$,
where $\overline{B(0,\delta)}=\{y\in\sH_c:|y-c|_{-1}>\delta\}$.
Since $P_t^{n,c}\to P_t^c$, it is sufficient to show that for such $t>0$, $P^{c,n}_t(x,\overline{B(0,\delta)})^c)\leq 1-\varepsilon$ with $\varepsilon>0$ independent by $n$.
Let $W_A(t,x)$ be the solution of 
\[
   \begin{cases}
     DZ=A^2Zdt+BdW(t),\\
    Z(0)=x.
   \end{cases}
\]
Setting $Y^n=X^n-W_A(t)$, we have that  $Y^n\in \sH_0$ and it sastisfies the equation
\[
 \begin{cases}
    dY^n(t)=A(-A Y^n+\lambda Y^n+f_n(Y^n+W_A))dt,\\
    Y^n(0)=0.
 \end{cases}
\]
By multiplying both sides by $(-A)^{-1}Y^n$ and integrating on $[-1,1]$ we find
\begin{eqnarray*}
 \frac{1}{2}d|Y^n|_{-1}^2&=&-\|\nabla Y^n\|_2^2+\lambda \|Y^n\|_2^2 +\langle f_n(Y^n+W_A),Y^n\rangle
\\
  &\leq& -\|\nabla Y^n\|_2^2+\lambda \|Y^n\|_2^2 +\langle f_n(W_A),Y^n\rangle
\\
  &\leq& -\|\nabla Y^n\|_2^2+\lambda \|Y^n\|_2^2 +\|f_n(W_A)\|_2^2+4\|Y^n\|_2^2
\end{eqnarray*}
where we have used Young's inequality and the fact $(f_n(a+b)-f_n(a))b\leq 0$ for any $a,b\in\R$ and $n\geq0$.
It is clear that there exists $C>0$ sucht that
\[
  -\|\nabla z\|_2^2+(\lambda +4)\|z\|_2^2\leq C|z|_{-1}^2 
\]
for any $z\in H^1$. Then we deduce
\[
  \frac{1}{2}d|Y^n|_{-1}^2\leq C|Y^n|_{-1}^2+\|f_n(W_A)-f_n(c)\|_2^2.
\]
So by Grownwall inequality we find
\[
  |Y^n(t)|_{-1}^2\leq \int_0^te^{2C(t-s)}\|f_n(W_A(s))\|_2^2ds.
\]
Consequently,
\begin{eqnarray}
 \P(|X^n(t)|_{-1}>\delta)   &\leq&    \P(|Y^n(t)|_{-1}+|W_A(t)|_{-1}>\delta)\leq \P(|Y^n(t)|_{-1}^2+|W_A(t)|_{-1}^2>\delta^2/2) \notag
\\
      &\leq&    \P\left(\int_0^te^{2C(t-s)}\|f_n(W_A(s))\|_2^2ds+|W_A(t)|_{-1}^2>\delta^2/2\right).\label{eq.fn}
\end{eqnarray}
By the properties of the approximating functions $f_n$ we have
$
     |f_n(W_A(s))|\leq |f(W_A(s))|
$
where the right-hand side is allowed to be $+\infty$ if $|W_A(t,x)|\geq 1$.
This implies that \eqref{eq.fn} is bounded by
\begin{equation*} 
  \P\left(\int_0^te^{2C(t-s)}\|f (W_A(s))\|_2^2ds+|W_A(t)|_{-1}^2>\delta^2/2\right).
\end{equation*}
In order to conclude the proof, it is sufficient to observe that by the gaussianity of the stochastic convolution this probability is $<1$. 
%
\end{proof}

\section*{Conclusion}

We obtain the ergodicity of the limit semi-group by assuming an essentially elliptic structure of the noise (that is the noise acts on the first modes).
We expect that these properties (expecially the ASF property) hold also with more general noises  but this need a better understanding of how the nonlinearity spreads the noise. 


Unfortunately the structure of the limit equation is unknown. In previous papers, for the same equation perturbed by white noise, it is proved that the nonlinearity has lead to reflection measures.
On the other side, in the deterministic case, there is a $L^\infty$ bound proving that the solution never approach the singularities.
In both cases they prove regularity in space that we have been unable to obtain in our degenerate noise case.
With enough regularity, we expect a good characterization of the solution (with or without reflection measure) leading to uniqueness.


\end{document}